\documentclass[a4paper,11pt]{article}
\usepackage{amsthm,amsmath,amssymb,amsfonts,mathrsfs,hyperref,microtype}
\usepackage{graphicx,color}
\usepackage[utf8]{inputenc}
\usepackage{epsfig}
\usepackage{wasysym}
\usepackage{mathrsfs}
\usepackage{graphicx,color}
\usepackage{eucal}
\usepackage{geometry}
\usepackage{pdflscape}

\theoremstyle{definition}
\newtheorem{defn}{Definition}[section]
  
\theoremstyle{plain}
\newtheorem{thm}{Theorem}[section]
\newtheorem{prop}{Proposition}[section]
\newtheorem{cor}{Corollary}[section]
\newtheorem{lma}{Lemma}[section]
\theoremstyle{remark}
\newtheorem{rmk}{Remark}[section]
\newtheorem{exm}{Example}[section]

\newcommand{\Z}{\mathbb{Z}}
\newcommand{\N}{\mathbb{N}}
\newcommand{\R}{\mathbb{R}}
\newcommand{\E}{\mathbb{E}}
\renewcommand{\P}{\mathbb{P}}

\newcommand{\e}{\mathrm{e}}

\newcommand{\ud}{\mathrm{d}}
\newcommand{\law}{\stackrel{\text{law}}{=}}

\numberwithin{equation}{section}

\font\eka=cmex10
\def\ind{\mathrel{\hbox{\rlap{%
\hbox to 7.5pt{\hrulefill}}\raise6.6pt\hbox{\eka\char'167}}}}

\begin{document}

\title{\textbf{Representation of stationary and stationary increment processes via Langevin equation and self-similar processes}}
\date{\today}




\renewcommand{\thefootnote}{\fnsymbol{footnote}}

\author{Lauri Viitasaari\footnotemark[1],\footnotemark[2]}

\footnotetext[2]{Department of Mathematics and System Analysis, Aalto University School of Science, Helsinki P.O. Box 11100, FIN-00076 Aalto,  Finland, {\tt lauri.viitasaari@aalto.fi}.}

\footnotetext[1]{Department of Mathematics, Saarland University, Post-fach 151150, D-66041 Saarbr\"ucken, Germany.}

\maketitle

\abstract{Let $W_t$ be a standard Brownian motion. It is well-known that the Langevin equation $\ud U_t = -\theta U_t\ud t + \ud W_t$ defines a stationary process called 
Ornstein-Uhlenbeck process. Furthermore, Langevin equation can be used to construct other stationary processes by replacing Brownian motion $W_t$ 
with some other process $G$ with stationary increments. In this article we prove that the converse also holds and all continuous stationary processes arise 
from a Langevin equation with certain noise $G=G_\theta$. Discrete analogies of our results are given and applications are discussed.

\



\section{Introduction}
Let $G=(G_t)_{t\in\R}$ be a continuous process with stationary increments and consider a Langevin equation
\begin{equation}
\label{langevin}
\ud U_t = - \theta U_t \ud t + \ud G_t, \quad t \in \R
\end{equation}
with some condition on $U_0$. If $G=W$ is a standard Brownian motion, then it is well-known that the equation \eqref{langevin} has a stationary solution that is called 
the Ornstein-Uhlenbeck process. Consequently, Langevin equation is connected to Ornstein-Uhlenbeck processes and can be used to construct stationary processes. Furthermore, 
Langevin type equations also have applications in statistical physics which highlights the importance of these equations even more. 

A natural question related to 
equation \eqref{langevin} is whether it has a stationary solution with suitably chosen initial condition for more general noise term $G$, 
and this question was studied 
recently in \cite{nie-con} for general stationary increment noise $G$. 
The main result in \cite{nie-con} was that under mild integrability assumptions on the driving force $G$, equation \eqref{langevin} has a stationary solution.

Another useful tool to construct stationary processes is via Lamperti theorem \cite{lamperti} which states that each $H$-self-similar process $X$ 
(for details on self-similar processes we refer to  monographs \cite{emb-mae,fla-bor-amb,tudor} dedicated to the subject) can be written as a 
Lamperti-transform $X=\mathcal{L}_H Y$, where $Y$ is a stationary process. Moreover, Lamperti-transform is invertible and hence stationary processes can be constructed 
from $H$-self-similar process via inverse transform $Y = \mathcal{L}_H^{-1}X$. Self-similar processes also have numerous applications. For example, 
the connection between self-decomposable laws and Levy processes (for details, we refer to \cite{bertoin,sato}) with self-similar processes is studied 
in \cite{jur-ver,wolfe,sato2,jean-pit-yor} to name a few. Especially, Jeanblanc et al. \cite{jean-pit-yor} used $H$-self-similar processes to connect two different 
representations derived in \cite{jur-ver,wolfe} and in \cite{sato2} for self-decomposable laws where the first representation is in terms of Levy processes 
(so called background driving Levy process) and the second representation is in terms of $H$-self-similar processes.

A process with special interest is the case of fractional Brownian motion with $H\in(0,1)$ and Ornstein-Uhlenbeck processes associated with it has been studied 
in \cite{b-c-s,c-k-m,k-s}. 
Recall that $B^H$ is the only Gaussian process which is $H$-self-similar and has stationary increments. Consequently, both approaches can be used to construct stationary processes. 
However, it is also known that the resulting processes are the same (in law) only in the case $H=\frac{1}{2}$, i.e. in the case of standard Brownian motion. 
On the other hand, it was proved by Kaarakka and Salminen \cite{k-s} that even the Lamperti transform of fractional Brownian motion can be defined as a solution to Langevin type 
equation with some driving noise $G$.  Furthermore, statistical problems for fractional Ornstein-Uhlenbeck processes have been studied at least in 
\cite{a-m,a-v2,b-i,h-n,k-l,x-z-x}. For research related to more general self-similar Gaussian processes, see also \cite{nu-po,yazigi}.

In this article we make use of the Lamperti theorem to characterise (continuous) stationary processes as solutions 
to the Langevin equation 
(\ref{langevin}) with some noise process $G$ belonging to a certain class $\mathcal{G}_H$. More precisely, as our main result we show that 
a process is stationary if and only if it is a solution to the Langevin equation with noise $G\in\mathcal{G}_H$. As a simple consequence it follows that Langevin equation 
(\ref{langevin}) has a stationary solution if and only if the noise process $G$ belongs to $\mathcal{G}_H$; a result which generalises the main findings of \cite{nie-con}. 
Moreover, we characterise the class $\mathcal{G}_H$ 
in terms of $H$-self-similar processes, and hence the results of this paper connects the two mentioned approaches to construct stationary processes in a natural way. 
We also present discrete analogies to our main theorems and consider some applications. 
For example, as a consequence of our main result we obtain that all discrete time stationary models reduces to a \emph{$AR(1)$-model} with a non-white noise. 

The rest of the paper is organised as follows. In section \ref{sec:pre} we introduce our notation and preliminary results, and in section \ref{sec:main} we present and prove our main results. Section \ref{sec:ex} is devoted to applications and examples.
\section{Notation and preliminaries}
\label{sec:pre}
Throughout the paper we assume that all processes $(X_t)_{t\in \R}$ have 
continuous paths almost surely (for extensions, see remark \ref{rmk:generalisation} below).  
Consequently, we can define integrals of form
$
\int_s^t e^{Hu}\ud X_u
$
over a compact interval $[s,t]$ and some constant $H\in \R$ via integration by parts formula
$$
\int_s^t e^{Hu}\ud X_u = e^{Ht}X_{t} - e^{Hs}X_{s} - H\int_s^t X_u e^{Hu}\ud u,
$$
where the last integral is understood as a Riemann integral. For numbers $t<s$ we use standard definition $\int_s^t = - \int_t^s$.
We also consider indefinite integrals of type
$
\int_{-\infty}^t e^{Hu}\ud X_u
$
which are defined similarly as 
$$
\int_{-\infty}^t e^{Hu}\ud X_u = e^{Ht}X_t - H \int_{-\infty}^t e^{Hu}X_u \ud u
$$
provided that the integral on the right exists almost surely.
\begin{rmk}
\label{rmk:generalisation}
We remark that for our purposes, the key ingredient is the fact that the above integrals under consideration can be defined as Riemann integrals and integration by parts 
is valid. Hence the assumption of almost sure continuity is not needed \emph{a priori} but it is made to ensure that the integrals considered can be understood pathwise 
without any additional technicalities. In comparison, in \cite{nie-con} the authors assumed integrability of the driving force of the Langevin equation which ensured 
that the integrals can be understood as $L^p$-limits of Riemann-Stieltjes sums. For generalisations, 
see also subsection \ref{subsec:additive}. 
\end{rmk}
We denote $X_t\law Y_t$ if finite dimensional distributions of $X$ and $Y$ are equal. Throughout the paper, we consider strictly stationary processes, i.e. processes $U=(U_t)_{t\in \R}$ for which $U_{t+h} \law U_t$ for every $h \in \R$. However, we remark that all the results of this paper are true for covariance stationary processes as well.

Next we recall definition of Lamperti transform and its inverse together with the famous Lamperti theorem. First we recall the definition of self-similar processes. 
\begin{defn}
Let $H>0$. A process $X=(X_t)_{t\ge 0}$ with $X_0=0$ is $H$-self-similar if
$$
X_{at} \law a^H X_t
$$
for every $a>0$. 
The class of all $H$-self-similar processes on $[0,\infty)$ are denoted by $\mathcal{X}_H$.
\end{defn}
\begin{defn}
Let $X=(X_t)_{t\geq 0}$ and $U=(U_t)_{t\in \R}$ be stochastic processes. We define
$$
(\mathcal{L}_H U)_t = t^H U_{\log t}, \quad t> 0
$$
and its inverse
$$
(\mathcal{L}_H^{-1}X)_t = e^{-Ht}X_{e^t},\quad t\in\R.
$$
\end{defn}
The result due to Lamperti \cite{lamperti} gives a one-to-one correspondence between stationary processes and $H$-self-similar processes.
\begin{thm}[Lamperti]
\label{thm:lamperti}
Let $U=(U_t)_{t\in\R}$ be a stationary process. Then $X = \mathcal{L}_H U$ is $H$-self-similar. Conversely, if $X=(X_t)_{t\geq 0}$ is $H$-self-similar, then $U = \mathcal{L}_H^{-1}X$ is stationary.
\end{thm}
\begin{rmk}
Note that for a given $H$-self-similar process $X$ the process
$$
U_t = e^{-\theta Ht}X_{e^{\theta t}},\quad t\in \R
$$
also defines a stationary process. However, for our purposes we only consider the case $\theta=1$.
\end{rmk}
We will consider the following class of processes.
\begin{defn}
Let $H>0$ be fixed and let $G=(G_t)_{t\in \R}$ be a stochastic process. We denote $G \in \mathcal{G}_H$ if $G_0=0$, $G$ has stationary increments, and 
\begin{equation}
\int_{-\infty}^0 e^{Hu}\ud G_u
\end{equation}
exists and defines an almost surely finite random variable.
\end{defn}
We will also give discrete versions of our results in which case we have the following analogous definition. 
\begin{defn}
Let $H>0$ be fixed and let $G=(G_n)_{n\in \N}$ be a stochastic process. We denote $G \in \mathcal{G}^d_H$ if $G_0=0$, $G$ has stationary increments, and
\begin{equation}
\lim_{k\rightarrow -\infty} \sum_{j=k}^0 e^{jH}\Delta G_j
\end{equation}
exists and defines an almost surely finite random variable, where $\Delta G_j = G_j - G_{j-1}$.
\end{defn}
It is not clear in general which stationary increment processes $G$ with $G_0=0$ belongs to $\mathcal{G}_H$ although this is the case provided that 
$
\lim_{s\rightarrow -\infty} e^{rs}|G_s| = 0
$
almost surely for some $r<H$. The following proposition shows that a mild integrability is sufficient for this purpose.
\begin{thm}
\label{lma:extension}
Let $G=(G_t)_{t \in \R}$ be a stationary increment process with $G_0=0$ such that 
\begin{equation}
\label{eq:integrability}
\sup_{t\in[0,1]}\E \left(\log|G_{t}|\textbf{1}_{\{|G_t|>1\}}\right)^{2+\delta} < \infty
\end{equation}
for some $\delta>0$, where $\textbf{1}_{A}$ denotes the indicator of a set $A$. Then $G\in\mathcal{G}_H$ for every $H>0$.
\end{thm}
\begin{proof}
Note that it is sufficient to prove that for every $H>0$ we have
\begin{equation}
\label{convergence}
\lim_{t\rightarrow \infty}e^{-Ht}|G_t| \rightarrow 0
\end{equation}
almost surely. Suppose first that $t=N$ is an integer. We have
$$
|G_N| \leq \sum_{k=1}^N |G_k-G_{k-1}|.
$$
Furthermore, it is clear that for each $\epsilon>0$ we have
$$
\P\left(\sum_{k=1}^N |G_k-G_{k-1}| > \epsilon\right)\leq \sum_{k=1}^N \P\left(|G_k-G_{k-1}|>\frac{\epsilon}{N}\right)
$$
and by stationarity of the increments we have
$$
\sum_{k=1}^N \P\left(|G_k-G_{k-1}|>\frac{\epsilon}{N}\right)= N \P\left(|G_1|>\frac{\epsilon}{N}\right).
$$
Consequently, we obtained
\begin{equation}
\label{eq:middle}
\P\left(e^{-HN}|G_N|>\epsilon\right) \leq N \P\left(|G_1|>\frac{e^{HN}\epsilon}{N}\right).
\end{equation}
Here
$$
\P\left(|G_1|>\frac{e^{HN}\epsilon}{N}\right) = \P\left(\log|G_1|>HN + \log\frac{\epsilon}{N}\right)
$$
and since $ HN + \log\frac{\epsilon}{N} \geq cN$ for some constant $c=c_{\epsilon}$ and large enough $N$, we get
$$
\P\left(\log|G_1|>HN + \log\frac{\epsilon}{N}\right) \leq \P\left(\log|G_1|>cN\right) = \P\left(\log|G_1|\textbf{1}_{\{|G_1|>1\}}>cN\right).
$$
Furthermore, we have
$$
\P\left(\log|G_1|\textbf{1}_{\{|G_1|>1\}}>cN\right) \leq C_{\epsilon}\frac{\E |\log|G_1|\textbf{1}_{\{|G_1|>1\}}|^{2+\delta}}{N^{2+\delta}}.
$$
In view of \eqref{eq:middle}, this implies
$$
\sum_{N=1}^\infty\P\left(e^{-HN}|G_n|>\epsilon\right) \leq C \sum_{N=1}^\infty N^{-1-\delta} < \infty
$$
for some constant $C>0$ and consequently, the claim follows by Borel-Cantelli Lemma. To conclude let $t$ be arbitrary and denote by $N$ the largest integer satisfying $N\leq t$. We have
$$
e^{-Ht}|G_t| \leq e^{-HN}|G_t - G_N| + e^{-HN}|G_N|.
$$
Now $e^{-HN}|G_N| \to 0$ almost surely by computations above, and the convergence 
$e^{-HN}|G_t - G_N|\to 0$ can be proved similarly.
\end{proof}
\begin{rmk}
Clearly we have analogous result for discrete processes.
\end{rmk}

Finally, the following proposition provides a counterexample which shows that not all stationary increment processes $G$ with $G_0=0$ belong to $\mathcal{G}_H$. 
We will present the counterexample only for discrete processes.
\begin{prop}
There exists a stationary increment process $G=(G_n)_{n\in\N}$ with $G_0=0$ which does not belong to $\mathcal{G}_H^d$ for any $H>0$.
\end{prop}
\begin{proof}
We construct an example of process $G_n$ such that 
$$
\lim_{k\to \infty} \sum_{j=0}^{k} e^{-Hj} \Delta G_j = \infty
$$
almost surely. 
Let $\alpha\in(0,1)$ and let $\xi_k$ be i.i.d. sequence of Pareto$(\alpha)$ random variables with tail function
$
\P (\xi_k > x) = x^{-\alpha}, \quad x>1.
$
Set
$
Z_k = e^{\xi_k}
$
and define
$
G_n = \sum_{k=1}^n Z_k.
$
We also set $G_0=0$. Now by defining two-sided process it is straightforward that $G_n$ has stationary increments.
On the other hand, now 
$$
\sum_{j=0}^{k} e^{-Hj} \Delta G_j = \sum_{j=0}^k e^{-Hj}Z_j,
$$
and since for any fixed number $A$ we have $\sum_{n=1}^\infty \P(e^{-Hn}Z_n >A) = \infty$, the series $\sum_{j=0}^k e^{-Hj}Z_j$ diverges with probability one 
by Kolmogorov's three series theorem \cite[IX.9, Theorem 3]{feller}.
\end{proof}

\section{Representation of stationary processes via Langevin equation}
\label{sec:main}
We begin with the following simple lemma.
\begin{lma}
\label{lemma_easy}
Let $X=(X_t)_{t\geq 0}$ be $H$-self-similar. Then the process
$$
Y_t = \int_0^t e^{-Hs}\ud X_{e^s}, \quad t\in \R,
$$
satisfies $Y \in \mathcal{G}_H$.
\end{lma}
\begin{proof}
Clearly, $Y_0=0$ and 
$
\int_{-\infty}^0 e^{Hs}\ud Y_t = \int_{-\infty}^0 \ud X_{e^s} = X_1.
$
Furthermore, for any $t,s,h\in\R$ we have
$$
Y_t - Y_s = \int_s^t e^{-Hu}\ud X_{e^u}
= \int_{s+h}^{t+h}e^{-H(v-h)}\ud X_{e^{v-h}}
$$
and by self-similarity of $X$, we have $
\ud X_{e^{-h}e^v} \law  e^{-hH}\ud X_{e^v}
$. 
Since integrals are defined as Riemann-integrals, we get immediately that $$Y_t - Y_s \law \int_{s+h}^{t+h}e^{-Hv}\ud X_{e^v} = Y_{t+h} - Y_{s+h}.$$ Treating 
$n$-dimensional vectors similarly proves the claim.
\end{proof}
Consider now the Langevin dynamics 
\begin{equation}
\label{langevin_dynamics}
\ud U_t = -H U_t \ud t + \ud G_t,\quad t\in \R.
\end{equation}
The unique solution can be expressed as 
$$
U_t = e^{-Ht}\left(U_0 + \int_0^t e^{Hs}\ud G_s\right),\quad t\in\R.
$$
In particular, if $G\in \mathcal{G}_H$, then for initial condition  
\begin{equation}
\label{eq:initial}
U_0 = \int_{-\infty}^0 e^{Hs}\ud G_s
\end{equation}
we get the solution
\begin{equation}
\label{U_def}
U_t = \e^{-H t}\int_{-\infty}^t e^{H s}\ud G_s, \quad t\in\R,
\end{equation}
and this process is stationary by the stationary increments of $G$. 
The main result of this paper shows that this fact is both necessary and sufficient for stationary processes.
\begin{thm}
\label{thm-main}
Let $H>0$ be fixed. A process $U=(U_t)_{t\in\R}$ is stationary if and only 
if it can be expressed as the unique solution to Langevin equation with some noise $G\in\mathcal{G}_H$ and initial condition \eqref{eq:initial}, i.e. 
\begin{equation}
\label{rep:U_G}
U_t = \e^{-H t}\int_{-\infty}^t e^{H s}\ud G_s,\quad t\in \R.
\end{equation}
Furthermore, the process $G\in\mathcal{G}_H$ in this representation is unique. 
\end{thm}
\begin{proof}
Define a process $X = \mathcal{L}_{H}U$. Now since $U$ is stationary, it follows by Lamperti theorem that $X$ is $H$-self-similar. Moreover, from $U_t = e^{-Ht}X_{e^t}$ we deduce 
$$
\ud U_t = -H U_t \ud t + e^{-H t}\ud X_{e^t}.
$$
By defining a process $Y=(Y_t)_{t\in\R}$ by
$$
Y_t = \int_{0}^t e^{-H s}\ud X_{e^s} 
$$
we have
$$
\ud U_t = - H U_t \ud t + \ud Y_t,
$$
and $Y \in \mathcal{G}_H$ by Lemma \ref{lemma_easy}. Conversely, for any $G\in\mathcal{G}_H$ the process \eqref{rep:U_G} corresponding to unique solution with initial 
condition \eqref{eq:initial}
is stationary. To show the uniqueness, let $H$ be fixed and assume there exists two processes $G^1\in\mathcal{G}_H$ and $G^2\in\mathcal{G}_H$ which yields same solution 
$U$ applied to (\ref{langevin_dynamics}). Hence
$$
e^{Ht}U_t = \int_{-\infty}^t e^{Hu}\ud G_u^1 = \int_{-\infty}^t e^{Hu}\ud G_u^2.
$$
In particular, for every $s<t$ we have
$$
\int_{s}^t e^{Hu}\ud G_u^1 = \int_s^t e^{Hu}\ud G_u^2
$$
which together with integration by parts yields
$$
e^{Ht}(G_t^1 - G_t^2) - e^{Hs}(G_s^1 - G_s^2) = H\int_s^t e^{Hu}(G_u^1-G_u^2)\ud u.
$$
Denoting $g(t) =e^{Ht}(G_t^1 - G_t^2)$ we obtain that for each fixed $s$ we have
$$
g(t) - g(s) = \int_s^t Hg(u)\ud u.
$$
Now it is well-known that the only solution to such equation is 
$
g(t) = Ce^{Ht}
$
which implies that there exists a constant $c$ such that $G_t^1 - G_t^2=c$ for every $t$. Applying this with $t=0$ we obtain $c=G_0^1 - G_0^2 = 0$. 
\end{proof}
As an immediate corollary we obtain the following result concerning existence of stationary solutions to Langevin equation.
\begin{cor}
\label{cor-existence}
For any $H>0$ the Langevin equation \eqref{langevin_dynamics} with a noise process $G=(G_t)_{t\in\R}$ satisfying $G_0=0$ 
has a stationary solution if and only if $G\in \mathcal{G}_H$. Furthermore, the stationary solution is unique in law.
\end{cor}
\begin{proof}
If $G\in \mathcal{G}_H$, then the existence of stationary solution follows from \eqref{rep:U_G}. Suppose next that a stationary solution exists with some 
initial value $U_0$ for the Langevin equation driven by $G$ satisfying $G_0=0$. 
Then by Theorem \ref{thm-main} the solution $U$ satisfies $\ud U_t = - H U_t \ud t + \ud Y_t$ for some $Y \in \mathcal{G}_H$. This implies $G=Y$ almost surely, 
and hence $G\in \mathcal{G}_H$. To conclude, the uniqueness in law can be deduced as in the proof of Theorem 2.1 in \cite{nie-con}.
\end{proof}
\begin{rmk}
The above result partially answers to the question raised in \cite{nie-con} whether logarithmic integrability is sufficient 
instead of $\E|G_t|^p < \infty, \quad t\geq 0$ for some $p\geq 1$ to quarantee the existence of stationary solution to the Langevin equation \eqref{langevin_dynamics}. 
Indeed, using Theorem \ref{lma:extension} we obtain that for any stationary increment process $G$ satisfying $G_0=0$ and 
\eqref{eq:integrability} the Langevin equation \eqref{langevin_dynamics} has a stationary solution.
\end{rmk}
\begin{cor}
\label{thm:representations}
Let $H>0$ be fixed. \begin{enumerate}
\item A process $X=(X_t)_{t\geq 0}$ is $H$-self-similar if and only if 
\begin{equation}
\label{ss-process}
X_t = \int_{-\infty}^{\log t} e^{Hs}\ud G_s, \quad t>0
\end{equation}
for some $G\in\mathcal{G}_H$. Furthermore, the process $G$ in the representation is unique.
\item A process $G =(G_t)_{t\in \R}$ satisfies $G\in\mathcal{G}_H$ if and only if it admits a representation
\begin{equation}
\label{si-rep}
G_t = \int_0^t e^{-Hu}\ud X_{e^u},\quad t\in\R
\end{equation}
for some $H$-self-similar process $X$. Furthermore, the process $X$ in this representation is unique.
\end{enumerate}
\end{cor}
\begin{proof}
\begin{enumerate}
\item By Theorem \ref{thm-main} there is one-to-one correspondence between stationary processes and processes $G\in \mathcal{G}_H$. 
Hence this item follows by setting $X_t = (\mathcal{L}_H U)_t = t^H U_{\log t}$ together with Lamperti theorem \ref{thm:lamperti}.
\item The fact that the process defined by \eqref{si-rep} belongs to $\mathcal{G}_H$ is 
the statement of Lemma \ref{lemma_easy}. For the converse, suppose $G\in\mathcal{G}_H$ and let $U$ be the unique stationary process arising from Langevin equation. By setting $X = \mathcal{L}_{H}U$ we have $
U_t = (\mathcal{L}_H^{-1}X)_t = e^{-Ht}X_{e^t}
$
and consequently, $U$ satisfies $\ud U_t = -HU_t \ud t + \ud Y_t$ 
with 
$Y_t = \int_0^t e^{-Hu} \ud X_{e^u}$. 
This implies $Y=G$ almost surely which concludes the proof.
\end{enumerate}
\end{proof}
We end this section by providing the following result for processes on $[0,\infty)$.
\begin{prop}
Let $U=(U_t)_{t\geq 0}$ be a stationary process and let $H>0$ be fixed. Then $U$ is a solution to the Langevin equation
\begin{equation}
\label{eq:langevin_once_more}
\ud U_t = - H U_t \ud t + \ud G_t, \quad t\geq 0
\end{equation}
with some stationary increment noise $G=(G_t)_{t\geq 0}$. 
Conversely, if there exists a two-sided process $\tilde{G}=(\tilde{G}_t)_{t\in\R} \in \mathcal{G}_H$ such that $\left(\tilde{G}_t\right)_{t\geq 0}\law (G_t)_{t\geq 0}$, 
then the Langevin equation \eqref{eq:langevin_once_more} has
a stationary solution with suitably chosen initial condition. In this case, the stationary 
solution is unique in law. 
\end{prop}
\begin{proof}
If $U=(U_t)_{t\geq 0}$ is stationary, then $X_t:=\left(\mathcal{L}_H U\right)_t = t^H U_{\log t}$ defines an $H$-self-similar process 
on $[1,\infty)$, and $U_t = e^{-Ht}X_{e^t}$. 
Consequently, $U$ satisfies \eqref{eq:langevin_once_more} with $G_t = \int_0^t e^{-Hs}\ud X_{e^s}$. Conversely, if the two-sided process $\tilde{G}\in \mathcal{G}_H$ exists, 
then the solution \eqref{rep:U_G} with $\tilde{G}$ defines a stationary process on whole $\R$ which satisfies Langevin equation with a driving force $\tilde{G}$. 
The existence of stationary solution follows from the fact that if driving forces $G^1$ and $G^2$ are equal in law, then so is 
solutions to corresponding equations \eqref{eq:langevin_once_more}.
\end{proof}
\subsection{Analogy in discrete time}
\label{sec:discrete}
Let $G\in \mathcal{G}^d_H$. Then the process defined by
\begin{equation}
\label{rep_U_discrete}
U_n = e^{-Hn}\sum_{k=-\infty}^n e^{Hk}\Delta_k G, \quad n\in \Z,
\end{equation}
where $\Delta_k G = G_k - G_{k-1}$, is well-defined and stationary process. Furthermore, it is straightforward to see that $U_n$ satisfies difference equation
\begin{equation}
\label{discrete_dynamics}
\Delta_n U = \left(e^{-H}-1\right)U_{n-1} + \Delta_n G,
\end{equation}
and hence difference equation (\ref{discrete_dynamics}) is a natural analogy to Langevin equation (\ref{langevin_dynamics}). 
Note also that a stationary process satisfying (\ref{discrete_dynamics}) corresponds to a 
\emph{$AR(1)$-model} with a noise $G$ which does not have independent increments. To define discrete analogy to the Lamperti transform 
$\mathcal{L}_HU_t = t^HU_{\log t}$, we define
$
M=\{m\in \R : \exists n\in \N s.t. m = e^n\}
$
and a process $X:(\Omega,M)\rightarrow \R$ by
$
X_m = e^{H\log m}U_{\log m}.
$
Clearly, for each $a,m\in M$ we have $am\in M$ and
$
X_{am} \law e^{H\log a}X_m.
$
In particular, we have
$
X_{e^{n+1}} - X_{e^n} \law e^H (X_{e^{n}} - X_{e^{n-1}}).
$
With these definitions we are able to give discrete analogy to our main theorem. The proof follows the same lines as in the continuous case and the details are left to the reader.
\begin{thm}
\label{thm-main-discrete}
Let $H>0$ be fixed. Then a process $(U_n)_{n\in \N}$ is stationary if and only if it can be expressed via equation (\ref{rep_U_discrete}) for some unique $G \in \mathcal{G}_H^d$.
\end{thm}
Since $U$ defined by \eqref{rep_U_discrete} satisfies \eqref{discrete_dynamics}, we immediately get the following corollary.
\begin{cor}
Every discrete time stationary process $(U_n)_{n\in \N}$ can be represented as an $AR(1)$-process.
\end{cor}
\begin{rmk}
Compared to the literature, stationary $ARMA(p,q)$ models and their extensions are widely applied and they have received a lot of attention. 
Furthermore, by famous Wold's decomposition theorem every discrete time stationary process can be viewed as $MA(\infty)$ model. 
According to previous corollary, every such model reduces to $AR(1)$ model with a noise which does not have independent increments. 
\end{rmk}

\section{Examples and applications}
\label{sec:ex}

\subsection{Additive self-similar processes} 
\label{subsec:additive}
A process $X$ is called additive if $X$ is stochastically continuous with cadlag paths (i.e. right-continuous with left-limits), $X$ has independent increments and $X_0=0$. 
If in addition $X$ has also stationary increments, then $X$ is a Levy process which is particularly important and widely applied class of processes. In \cite{jean-pit-yor} 
the authors studied $H$-self-similar additive processes $X$ and proved that such $X$ can be represented as an indefinite integral with respect to a Levy process. Furthermore, 
as a corollary (Corollary 2 in \cite{jean-pit-yor}) the authors obtained that the inverse Lamperti transform of an additive $H$-self-similar process $X$ can be described as a 
solution to a Langevin equation driven by a Levy process. Since the integrals with respect to Levy processes can be defined and they satisfy the key integration by parts 
formula presented 
in this paper, one can easily check that the results in \cite{jean-pit-yor} can be recovered from the main theorems presented in this paper applied to a special case. 
Indeed, if $X$ is $H$-self-similar process additive process it is clear from representation 
(\ref{si-rep}) that the corresponding process $G$ is also an additive process with stationary increments, and hence a Levy process. Conversely, given a Levy process $G$ it 
is straightforward to see from representation (\ref{ss-process}) that the corresponding process $X$ defines an $H$-self-similar additive process. Furthermore, representation 
of $\mathcal{L}_H^{-1}X$ as a solution to a Levy-driven Langevin equation (Corollary 2 of \cite{jean-pit-yor}) follows immediately from Theorem \ref{thm-main}.
\subsection{Gaussian processes}
It is known that a continuous time stationary Gaussian process $U_t$ is either continuous or unbounded on every interval $[a,b]$ (see, e.g. \cite{adler}) and the latter case is 
hardly interesting. Furthermore, every continuous Gaussian process $G\in\mathcal{G}$ is also an element of $\mathcal{G}_H$ for every $H>0$. Consequently, every continuous 
time stationary Gaussian process $U_t$ can be represented as the unique solution to Langevin equation with some noise term $G\in\mathcal{G}_H$. 
Conversely, given any Gaussian process $G\in\mathcal{G}$ and given any fixed $H>0$, the Langevin equation determines a unique stationary Gaussian process.
\begin{exm}
Consider Lamperti transform of standard Brownian motion defined by
$$
U_t = \frac{1}{\sqrt{2\theta}}e^{-\theta t}W_{e^{2\theta t}}.
$$
Now $U_t$ can be described as a solution to equation
$$
\ud U_t = -\theta U_t \ud t + \ud G_t
$$
with noise term
\begin{equation}
\label{bm_noise}
G_t = \frac{1}{\sqrt{2\theta}}\int_0^t e^{-\theta u}\ud W_{e^{2\theta u}}.
\end{equation}
Consequently, the Lamperti transform of standard Brownian motion is a pathwise solution to Langevin equation with noise given by (\ref{bm_noise}). On the other hand, it is straightforward to see that $G_t - G_s \law W_t - W_s$. 
\end{exm}
\subsubsection{fractional Ornstein-Uhlenbeck processes}
Recall that a fractional Brownian motion $B^H = (B^H_t)_{t\geq 0}$, one of the best studied Gaussian process, is a centered Gaussian process with covariance
$$
R(s,t) = \frac12\left[s^{2H} + t^{2H} - |t-s|^{2H}\right].
$$
The process $B^H$ is $H$-self-similar and has stationary increments, and the case $H=\frac12$ corresponds to a standard Brownian motion. Fractional Brownian motion $B^H$ can be 
used as a driving force $G$ of the Langevin equation as well as for the associated self-similar process $X$, and the corresponding 
processes are called fractional Ornstein-Uhlenbeck processes. The one arising from Langevin equation was studied in details in \cite{c-k-m}, 
and the corresponding stationary process is called the fractional Ornstein-Uhlenbeck process \emph{of the first kind} by using the terminology introduced in \cite{k-s}. Using 
Corollary \ref{thm:representations} we can deduce that stationary fractional Ornstein-Uhlenbeck process of the first kind can be viewed as an inverse Lamperti transform of $H$-self-similar 
process
$$
X_t = \int_{-\infty}^{\log t}e^{Hs}\ud B^H_s = t^H B^H_{\log t} - H\int_{-\infty}^{\log t}e^{Hs}B^H_s\ud s.
$$
Similarly in \cite{k-s}, the motivation for the authors work was to study transform 
$$
X_t^{(\alpha,H)} = e^{-\alpha t}B_{a_t}^H,
$$
where $a_t = \frac{H}{\alpha}e^{\frac{H}{\alpha}t}$ and $\alpha>0$ is a constant. This process is stationary, and $\alpha=H$ corresponds to the inverse Lamperti transform  
$\mathcal{L}^{-1}_H B^H$. It was proved in \cite{k-s} that $X^{\alpha,H}$ can be viewed as a solution to a Langevin equation
$
\ud X_t^{(\alpha,H)} = - \alpha X_t^{(\alpha,H)} \ud t + \ud Y_t^{(\alpha,H)}
$ 
with a noise process 
$
Y_t^{(\alpha,H)} = \int_0^t e^{-\alpha s}\ud B^H_{a_t},
$
and inspired by a scaling property 
$$
\left(\alpha^H Y_{\frac{t}{\alpha}}^{(\alpha,H)}\right)_{t\geq 0} \law \left(Y_t^{(1,H)}\right)_{t\geq 0}
$$
the authors introduced the fractional Ornstein-Uhlenbeck process of \emph{the second kind} as the unique stationary solution to the Langevin equation
$$
\ud U_t^{\theta} = - \theta U_t^{\theta} \ud t  + \ud Y_t^{(1,H)}.
$$
Now with some simple computations together with Corollary \ref{thm:representations} we obtain that $U_t^{\theta}$ is an inverse Lamperti transform of a 
$\theta$-self-similar process 
$$
X_t^{(\theta,H)} = \int_0^{Ht^{\frac{1}{H}}} \left(\frac{u}{H}\right)^{H(\theta-1)} \ud B_u^H
$$
which unfortunately is not so closely related to the inverse Lamperti transform of fractional Brownian motion except the fact that 
this process shares some useful properties with $\mathcal{L}^{-1}B^H$. For further details on the fractional Ornstein-Uhlenbeck process of the second kind, we 
refer to \cite{k-s,a-m,a-v2}.
\subsection{$ARMA(p,q)$-models}
As a discrete time example, consider a general $ARMA(p,q)$-model (for details on time series, we refer to \cite{hamilton}) defined by
$$
X_n = c + \sum_{k=1}^p \alpha_k X_{n-k} + \sum_{k=1}^q \beta_k \xi_{n-k},
$$
where the sequence $\xi_{n-k}$ is a white noise. By Theorem \ref{thm-main-discrete}, 
stationary $ARMA(p,q)$ process $X_n$ can be uniquely represented as an \emph{$AR(1)$ model} as
$
X_n = e^{-H} X_{n-1} + \Delta_k G
$
with some parameter $H$ and stationary increment noise $G$ which do not have independent increments. 
Similarly, by Wold's representation Theorem every discrete time covariance stationary process can be represented as $MA(\infty)$ model with representation
$$
X_n = \sum_{j=0}^\infty b_j \xi_{n-j} + \eta_n,
$$
where $\xi_{n-j}$ is the stochastic innovation process and $\eta_n$ is deterministic. 
Now, thanks to Theorem \ref{thm-main-discrete}, such process can equivalently be written as \emph{$AR(1)$ model} with representation 
(\ref{rep_U_discrete}). In a similar way, all stationary generalisations of $ARMA(p,q)$ reduce back to \emph{$AR(1)$ model}.

\subsection*{Acknowledgments}
The author thanks Tommi Sottinen and Lasse Leskel\"a for careful reading of the manuscript and their many helpful comments. This research was financed by Aalto University Department of Mathematics and Systems Analysis.

\bibliographystyle{plain}
\bibliography{bibli_v2}
\end{document}